\documentclass[12pt]{amsart}
\usepackage{amssymb,enumerate,mathrsfs, amsmath, amsthm}
\usepackage{url,titletoc,enumitem}

\usepackage[usenames,dvipsnames]{xcolor}
\definecolor{darkblue}{rgb}{0.0, 0.0, 0.55}
\usepackage[pagebackref,colorlinks,linkcolor=BrickRed,citecolor=OliveGreen,urlcolor=darkblue,hypertexnames=true]{hyperref}
\usepackage{cmap}

\renewcommand{\subset}{\subseteq}
\renewcommand{\emptyset}{\varnothing}

\newtheorem{theorem}{Theorem}[section]

\newtheorem{lemma}[theorem]{Lemma}

\newtheorem{prop}[theorem]{Proposition}

\newtheorem{alg}[theorem]{Algorithm}
\newtheorem{remark}[theorem]{Remark}

\newtheorem{thm}[theorem]{Theorem}
\newtheorem{lem}[theorem]{Lemma}
\newtheorem{Cor}[theorem]{Corollary}

\newtheorem*{lemma*}{Lemma}

\def\beq{\begin{equation}}
\def\eeq{\end{equation}}

\numberwithin{equation}{section}

\def\beq{\begin{equation}}
\def\eeq{\end{equation}}

\def\cD{ {{\mathcal D}}}
\def\cH{ {{\mathcal H}}}

\def\bbN{ {\mathbb N}}
\def\bbS{ {\mathbb S}}

\def\R{ {\mathbb{R}} }
\def\C{ {\mathbb{C}} }

\def\bbS{{\mathbb S}}

\def\cD{ {\mathcal D} }

\def\cH{ {\mathcal H} }

\def\cM{{\mathcal M}}

\def\cR{{ \mathcal R }}

\def\beq{\begin{equation}}
\def\eeq{\end{equation}}

\def\cR{{\mathcal R}}

\newcommand{\df}[1]{{\bf{#1}}{\index{#1}}}

\def\abs{\partial^{\mathrm{abs}}}

\def\maxK{K}

\def\maxRKX{R_{K_X}}

\title{Matrix convex sets without absolute extreme points}

\author[E. Evert]{Eric Evert${}^1$}
\address{Eric Evert, Department of Mathematics\\
  University of California \\
  San Diego
   }
   \email{eevert@ucsd.edu}
\thanks{${}^1$Research supported by the NSF grant
DMS-1500835}
\makeindex

\begin{document}
\maketitle

\begin{abstract}
This article shows the existence of a class of closed bounded matrix convex sets which do not have absolute extreme points. The sets we consider are noncommutative sets, $K_X$, formed by taking matrix convex combinations of a single tuple $X$. In the case that $X$ is a tuple of compact operators with no nontrivial finite dimensional reducing subspaces, $K_X$ is a closed bounded matrix convex set with no absolute extreme points.

A central goal in the theory of matrix convexity is to find a natural notion of an extreme point in the dimension free setting which is minimal with respect to spanning. Matrix extreme points are the strongest type of extreme point known to span matrix convex sets; however, they are not necessarily the smallest set which does so. Absolute extreme points, a more restricted type of extreme points that are closely related to Arveson's boundary, enjoy a strong notion of minimality should they span. This result shows that matrix convex sets may fail to be spanned by their absolute extreme points.
\end{abstract}

\section{Introduction}

One of the central topics in matrix convexity is the subject of extreme points. In the dimension free setting there are many notions of an extreme point. One class, introduced by Webster and Winkler in  \cite{WW99}, is the notion of a matrix extreme point. The main result in \cite{WW99} shows that a closed bounded matrix convex set $K$ is spanned by its matrix extreme points, i.e. the closed matrix convex hull of the matrix extreme points of $K$ is equal to $K$, and is a critical result in the theory of matrix convex sets. However, it is often the case that a proper subset of the matrix extreme points spans $K$. In fact, a matrix convex combination involving a single matrix extreme point of $K$ can produce a new matrix extreme point \cite{A69,F00,F04}. As of today, it remains unknown if there is a natural notion of extreme points for matrix convex sets which is minimal with respect to spanning.

A more restricted class of extreme point is the notion of an absolute extreme point which was introduced by Kleski \cite{KLS14}. This class of extreme points is closely related to Arveson's notion \cite{A69} of a boundary representation of an operator system \cite{KLS14,EHKM16}. In the setting of matrix convex sets in finite dimensions the primary distinction between these notions is that boundary representations may be infinite dimensional objects while absolute extreme points, being elements of a matrix convex set, are finite dimensional objects. If one extends a matrix convex set to the corresponding operator convex set which contains infinite dimensional tuples, then these notions coincide. For a recent account of results related to extreme points of matrix convex sets see \cite{EHKM16}.

 When Arveson first proposed the existence of boundary representations, he conjectured that each operator system had sufficiently many boundary representations to completely norm it. Much work has gone into this question in the subsequent years.

 Hanama \cite{Ham79} took the first major step towards a positive answer by showing the existence of the $C^*$-envelope, the $C^*$-algebra generated by the would be boundary representations. Many years later Dritschel and McCullough \cite{DM05} took another important step by introducing and proving the existence of maximal dilations of general representations. Using this dilation theoretic approach, Dritschel and McCullough gave a new proof of the existence of the $C^*$-envelope.

 After the result of Dritschel and McCullough, Arveson \cite{A08} showed the equivalence of boundary representations and irreducible maximal representations and was able to prove that, in the seperable case, an operator system is completely normed by its boundary representations. A few years later Arveson's question was put to rest when Davidson and Kennedy \cite{DK15} proved in the general setting using a dilation theoretic argument that every operator system has sufficiently many boundary representations to completely norm it. For additional material related to boundary representations and the $C^*$-envelope we direct the reader to \cite{MS98,FHL+}.

Although the situation is well understood in Arveson's infinite dimensional setting, in the finite dimensional setting it has remained unknown if matrix convex sets are spanned by their absolute extreme points. In this article we provide a negative answer to this question by giving a class of closed bounded matrix convex sets which do not have absolute extreme points.

For a finite collection $\{X^1, \dots, X^k\}$ of irreducible $g$-tuples of self-adjoint operators, call its noncommutative convex hull a \df{noncommutative hull polygon} with \df{corners} $X^1, \dots, X^k$. If $K$ is a noncommutative hull polygon and there is a finite matrix convex combination of its corners which is equal to $0 \in \R^g$, then we say $0$ is in the finite interior of $K$. The following corollary gives a partial statement of our main result, Theorem \ref{thm:NoAbsExt}.
\begin{Cor}
\label{Cor:HullPolygon}
Let $K$ be a noncommutative hull polygon with corners $X^1, \dots, X^k$ that are each irreducible $g$-tuples of compact operators acting on a separable infinite dimensional Hilbert space. Also assume that $0$ is in the finite interior of $K$. Then $K$ has no absolute extreme points.
\end{Cor}

In the remainder of this section we introduce our basic definitions and notation. Theorem \ref{thm:NoAbsExt} is stated at the end of the section.

\subsection{Notation and Definitions}
Throughout the paper we let $\cH$ be a separable Hilbert space and we take $(\cH_n)_n$ to be a nested sequences of subspaces of $\cH$ such that
\[
\dim(\cH_n)=n \mathrm{\ for\ all\ } n \in \bbN \mathrm{\ and\ } \cH=\overline{\cup_n \cH_n}
 \]
 where the closure is in norm. For any Hilbert space $\cM$, we use the notation $\cM^d=\oplus_1^d \cM$ where $d\in \bbN \cup \{ \infty\}$. We use $B(\cM)^g, \bbS(\cM)^g,$ and $\mathfrak{K} (\cM)^g$ to denote the sets
 of $g$-tuples of bounded operators, bounded self-adjoint operators, and compact self-adjoint operators on $\cM$, respectively. Similarly, given Hilbert spaces $\cM_1, \cM_2$, we let $B(\cM_1,\cM_2)$ be the set of bounded operators mapping $\cM_1 \to \cM_2$.

 Given tuples $Y, Z \in B(\cM)^g$ say $Y$ and $Z$ are \df{unitarily equivalent}, denoted by $Y \sim_u Z$, if there exists a unitary $U:\cM \to \cM$ such that
 \[
 U^*YU=(U^* Y_1 U, \dots, U^* Y_g U)=Z.
 \]

The set $\Gamma \subset \cup_n \bbS(\cH_n)^g$ is a sequence $\Gamma=(\Gamma(n))_n$ where $\Gamma(n) \subset \bbS(\cH_n)^g$ for all $n \in \bbN$. We say $\Gamma$ is \df{closed with respect to direct sums} if for any pair of positive integers $n,m \in \bbN$ and tuples $Y \in  \Gamma(n)$ and $Z \in \Gamma(m)$, the tuple $Y \oplus Z \in \Gamma (n+m)$. Here
\[
Y \oplus Z= (Y_1 \oplus Z_1, \dots, Y_g \oplus Z_g).
\]
Similarly, $\Gamma$ is \df{closed with respect to unitary conjugation} if, for each $n \in \bbN$ and $Y \in \Gamma(n)$ and each unitary $U\in B(\cH_n)$ we have
\[
U^*Y U \in \Gamma(n).
\]
If $\Gamma \subset \cup_n \bbS(\cH_n)^g$ is closed with respect to direct sums and unitary conjugation then $\Gamma$ is a \df{free set}. We say $\Gamma \subset \cup_n \bbS(\cH_n)^g$ is \df{bounded} if there is an $C \in \R^+$ such that $C-\sum_{i=1}^g Y_i^2 \succeq 0$ for all $Y \in \Gamma$. If $\Gamma(n)$ is closed for all $n \in \bbN$, then we say $\Gamma=(\Gamma(n))_n$ is \df{closed}.

\subsection{Matrix Convex Sets} Let $K \subset \cup_n \bbS(\cH_n)^g$.  A \df{matrix convex combination} of elements of $K$ is a finite sum of the form
\[
\sum_{i=1}^k V_i^* Y^i V_i
\]
where $Y^i \in K(n_i)$ for $i=1, \dots, k$ and $\sum_{i=1}^k V_i^* V_i =I_n$. If additionally $V_i \neq 0$ for each $i$ then the sum is said to be \df{weakly proper}. If $K$ is closed under matrix convex combinations  then $K$ is \df{matrix convex}. It should be noted that a matrix convex set is automatically a free set. Furthermore, to show a free set $K$ is matrix convex, it is sufficient to show that $K$ is closed under \df{isometric conjugation}, i.e. for any $m \leq n \in \bbN$, any $Y \in K(n)$ and any isometry $V:\cH_m \to \cH_n$, the tuple $V^* Y V \in K(m)$.

Given a matrix convex set $K$, say $Y \in K(n)$ is an \df{absolute extreme point} of $K$ if whenever $Y$ is written as a weakly proper matrix convex combination $Y=\sum_{i=1}^k V_i^* Z^i V_i$, then for all $i$ either $n_i=n$ and $Y \sim_u Z^i$ or $n_i > n$ and there exists $A^i \in K$ such that $Y \oplus A^i \sim_u Z^i$.

\subsection{Linear Pencils and Free Spectrahedra}

One class of examples of matrix convex sets are free spectrahedra, the solution sets to linear matrix inequalities. Let $d \in \bbN$ be a positive integer and let
\[
A=(A_0, A_1, \dots A_g) \in B(\cH_d)^{g+1}
\]
 be a $g+1$-tuple of finite dimensional operators on $\cH_d$. A \df{linear pencil}, denoted by $L_A$, is the map $y \mapsto L_A(y)$ defined by
\[
L_A (y)=A_0-A_1 y_1 - \dots - A_g y_g.
\]
We also allow $A \in B(\cH)^{g+1}$ to be a $g+1$-tuple of operators on $\cH$. In this case we say that $L_A$ is a \df{linear operator pencil}.
If $A$ is a $g+1$-tuple of self-adjoint operators, we will say $L_A$ is a \df{symmetric} linear (operator) pencil.

Given a positive integer $n$ and a tuple $Y \in B(\cH_n)^g$ for some $n$ and a linear (operator) pencil $L_A$ the \df{evaluation} of $L_A$ on $Y$ is defined by
\[
L_A(Y)=A_0 \otimes I_n-A_1 \otimes Y_1 -\dots - A_g \otimes Y_g.
\]

We are often interested in the case where $A_0=I$. In this case, given a $g$-tuple $A=(A_1, \dots, A_g)$ of operators on $B(\cH)^g$ or $B(\cH_d)^g$ for some  $d \in \bbN$, we use $\mathcal{L}_A$ to denote the \df{monic} linear (operator) pencil
\[
\mathcal{L}_A (y)=L_{(I,A)}(y)=I-A_1 y_1- \dots -A_g y_g.
\]

Let $A \in \bbS(\cH_d)^g$ be a $g$-tuple of self-adjoint operators and let $\mathcal{L}_A$ be the associated symmetric monic linear pencil. For each $n \in \bbN$ we define the \df{free spectrahedron at level $n$}, denoted $\cD_A (n)$, by
\[
\cD_A (n)=\{Y \in \bbS(\cH_n)^g|\ \mathcal{L}_A(Y) \succeq 0\}.
\]
The corresponding \df{free spectrahedron} is the sequence $\cD_A=(\cD_A(n))_n$. If $\mathcal{L}_A$ is a symmetric monic linear operator pencil, i.e. $A \in \bbS(\cH)^g$, we will say $\cD_A$ is an \df{operator free spectrahedron}. It is straight forward to show that (operator) free spectrahedra are matrix convex sets.

 \subsection{Completely Positive Maps}
We will assume the reader's familiarity with operator systems and completely positive maps. For a comprehensive discussion of these subjects see \cite{P02}.

Let $\cR$ be an operator system and let $\phi: \cR \to B(\cM_1)$ be a unital completely positive map for some Hilbert space $\cM_1$. A \df{dilation} of $\phi$ is a unital completely positive map of the form $\psi:\cR \to B(\cM_2)$ such that $\cM_2$ is a Hilbert space containing $\cM_1$ and $\iota^*_{\cM_1} \psi (r) \iota_{\cM_1}=\phi(r)$ for all $r \in \cR$. Here $\iota_{\cM_1}:\cM_1 \to \cM_2$ is the inclusion map of $\cM_1$ into $\cM_2$. A unital completely positive map $\phi$ is called \df{maximal} if any dilation $\psi$ of $\phi$ has the form $\psi=\phi \oplus \psi'$ for some unital completely positive $\psi'$.

We use the notation
\[
CS_n (\cR)=\{\phi:\cR \to B(\cH_n)|\ \phi \mathrm{\ is\ u.c.p.}\}
\]
 to denote the set of unital completely positive maps sending $\cR$ into $B(\cH_n)$ and we define
 \[
 CS(\cR)=\cup_n CS_n (\cR)
 \]
 to be the set of unital completely positive maps on $\cR$ with finite dimensional range.

\subsection{Main Result}

The main result of this paper is Theorem \ref{thm:NoAbsExt} which gives a class of closed bounded matrix convex sets each of which has no absolute extreme points. Our candidate sets are each noncommutative convex hulls, sets we now define.

Let $X \in \bbS(\cH)^g$ and for each $n \in \bbN$ define the set $K_X (n) \subset \bbS(\cH_n)^g$ by
\beq
\label{eq:KxDef}
K_X (n)=\{Y \in \bbS(\cH_n)^g | Y=V^* (I_\cH \otimes X)V  \mathrm{ \ for \ some \ isometry \ } V:\cH_n \to \oplus_1^\infty \cH \}.
\eeq
We then define $K_X \subset \cup_n \bbS(\cH_n)^g$ by
\beq
\label{eq:KxDefPart2}
K_X=\cup_n K_X (n).
\eeq
We call $K_X$ the \df{noncommutative convex hull} of X.

Given a $g$-tuple $X$, we say \df{$0$ is in the finite interior of $K_X$} if there exist an integer $d\in \bbN$ and a unit vector $v \in (\cH)^d$ such that
\[
v^* (I_d \otimes X ) v=0 \in \R^g.
\]

 We remark that our notion of the noncommutative convex hull of $X$ is closely related to Arveson's notion of the matrix range of $X$ \cite{A72}, the primary distinction being that a noncommutative convex hull may fail to be closed. In addition to \cite{A72}, see \cite{DDSS17} for a discussion of matrix ranges.

\begin{thm}
\label{thm:NoAbsExt}

Let $X \in \mathfrak{K}(\cH)^g$ be a $g$-tuple of compact self-adjoint operators on $\cH$ and let $K_X$ be the noncommutative convex hull of $X$. Assume that $X$ has no nontrivial finite dimensional reducing subspaces and assume $0$ is in the finite interior of $K_X$. Then $K_X$ is a closed bounded matrix convex set and $\abs K_X= \emptyset$.

\end{thm}

We now give the proof of Corollary \ref{Cor:HullPolygon}.

\begin{proof}
Set $X=\oplus_{i=1}^k X^i$. Then $X$ is an $g$-tuple of compact operators which has no nontrivial finite dimensional reducing subspaces. Furthermore, $0$ is in the finite interior of $K_X$ by assumption. Theorem \ref{thm:NoAbsExt} completes the proof.
\end{proof}

\subsection{Guide to the Paper}

The paper is organized in the following manner. In Section \ref{section:Closed} we show that $K_X$ is a bounded matrix convex set for any tuple $X$. We then show that such a set is closed provided that $X$ is compact and $0$ is in the finite interior of $K_X$.

Sections \ref{section:MatrixAffine} and \ref{section:NoAbsExtreme} show that the set $K_X$ has no absolute extreme points if $X$ is compact and has no nontrivial finite dimensional reducing subspaces. This is accomplished by introducing an operator system $\maxRKX$ such that $CS(\maxRKX),$ the set of unital completely positive maps on $\cR_{K_X}$ with finite dimensional range, is matrix affine homeomorphic to $K_X$. In particular, we show that there are no maximal unital completely positive maps in $CS(\cR_{K_X})$.

 Section \ref{section:MatrixAffine} introduces the notion of matrix affine maps and discusses the equivalence between the Arveson extreme points of $K_X$ and the irreducible maximal completely positive maps in $CS(\maxRKX)$. Section \ref{section:NoAbsExtreme} then completes the proof of Theorem \ref{thm:NoAbsExt} by showing that, if $X$ has no nontrivial finite dimensional reducing subspaces, then there are no maximal completely positive maps in $CS(\maxRKX)$.

 The paper ends with Section \ref{section:ExampleX} where we give an explicit example of a tuple $X$ which satisfies the assumptions of Theorem \ref{thm:NoAbsExt}.

\section{Matrix Convex Sets which are Matrix Convex Combinations of a Compact Tuple}
\label{section:Closed}

Our first objective is to show that noncommutative convex hulls are bounded matrix convex set and that there are reasonable assumptions which can be made on $X$ such that $K_X$ is closed. The main result of this section is Theorem \ref{theorem:KxClosedBoundedMatConvex} which shows that $K_X$ is a closed bounded matrix convex set when $X$ is a tuple of compact operators and $0$ is in the finite interior of $K_X$. We begin the section with a lemma which shows that $K_X$ is a bounded matrix convex set for any $X \in \bbS(\cH)^g$.

\begin{lem}
\label{lem:KxBoundedMatConvex}
Let $X \in \bbS(\cH)^g$ be a $g$-tuple of self-adjoint operators on $\cH$ and let $K_X$ be the noncommutative convex hull of $X$. Then $K_X$ is a bounded matrix convex set.
\end{lem}
\begin{proof}
To see $K_X$ is bounded, observe that for any $n$ and any isometry $V:\cH_n \to \cH^\infty$ we have the inequality
\[
\|V^*(I_\cH \otimes X)V\| \leq \|(I_\cH \otimes X)\|= \|X\|.
\]

 Furthermore, it is straight forward to show that $K_X$ is closed under direct sums. Since $K_X$ is closed under isometric and unitary conjugation by definition, it follows that $K_X$ is matrix convex.

\end{proof}

We now aim to prove that $K_X$ is closed when $X$ is compact and $0$ is in the finite interior of $K_X$. Proposition \ref{prop:FixedBoundOnXSize} shows that, with these assumptions, for each fixed $n$ the set $K_X (n)$ can be defined as the set of compressions of a tuple of compact operators and is the key result in proving that $K_X$ is closed.

\def\ktuples{{\mathfrak{K}(\cH)^g}}

\begin{prop}
\label{prop:FixedBoundOnXSize}
Let $X \in \ktuples$ be a $g$-tuple of self-adjoint compact operators on $\cH$.
\begin{enumerate}
\item
\label{it:FiniteSumContraction}
For each $n \in \bbN$ there exists an integer $m_n$ depending only on $n$ and $g$ such that for all $Y \in K_X (n)$ there is a contraction $W:\cH_n \to \cH^{m_n}$ such that $Y=W^* (I_{m_n} \otimes X) W$.

\item
\label{it:FiniteSumIso}
Assume $0$ is in the finite interior of $K_X$. Then there exists an integer $m_{0}$ depending only on $g$ such that for all $Y \in K_X (n)$ there is an isometry $T:\cH_n \to \cH^{m_n+nm_{0}}$ such that $Y=T^* (I_{m_n+nm_{0}} \otimes X) T$. In particular we have
\beq
\label{eq:KxFiniteSum}
K_X (n)=\{Y \in \bbS(\cH_n)^g | \ Y=T^* (I_{m_n+nm_{0}} \otimes X)T  \mathrm{ \ for \ some \ isometry \ } T:\cH_n \to  \cH^{m_n+nm_{0}}\}.
\eeq
\end{enumerate}
\end{prop}

Before giving the proof of Proposition \ref{prop:FixedBoundOnXSize} we state two lemmas which will be useful in the proofs of Proposition \ref{prop:FixedBoundOnXSize} and Theorem \ref{theorem:KxClosedBoundedMatConvex}. The first lemma is a convergence argument, while the second lemma shows that, when $0$ is in the finite interior of $K_X$, contractive conjugation can be replaced with isometric conjugation.

\begin{lemma}
\label{lem:XCompactConvergence}
Let $m,n \in \bbN$ be positive integers. Let $Z \in \mathfrak{K}(\cH^m)^g$ be a $g$-tuple of compact self-adjoint operators on $\cH^m$ and let $\{W^\ell\}_{\ell=1}^\infty \subset B(\cH_n,\cH^m)$ be a sequence of contractions which converges in the weak operator topology on $B(\cH_n,\cH^m)$ to some operator $W \in B(\cH_n,\cH^m)$. Then the sequence $\{W_\ell^* Z W_\ell\}_{\ell=1}^\infty \subset \bbS(\cH_n)^g$ converges in norm to $W^* Z W \in \bbS(\cH_n)^g$.

\end{lemma}

\begin{proof}
 Since $WZW$ and $W_\ell^* Z W_\ell$ for all $\ell$ are all tuples of finite dimensional operators it is sufficient to show that for each $i=1, \dots, g$ and for all $h \in \cH_n$ we have
\beq
\label{eq:ConvergenceInnerProduct}
\lim \|\langle (W_\ell^* Z_i W_\ell-W^* Z_i W)h,h \rangle\|=0.
\eeq

To this end, fix $h \in \cH_n$ and observe that
\beq
\label{eq:SplitSum}
\begin{array}{rclcl}
\lim \|\langle (W_\ell^* Z_i W_\ell-W^* Z W)h,h \rangle\| &\leq& \lim \|\langle W_\ell^* Z_i (W_\ell-W)h,h \rangle\|\\
&+&\lim \|\langle(W_\ell^*-W^*) Z_i W h,h \rangle\|
\end{array}
\eeq
Note that the sequence $\{W_\ell^*\}$ converges to $W^*$ in the strong operator topology since these operators map into a finite dimensional space. This implies that
\beq
\label{eq:StrongConvPart}
\lim \|\langle(W_\ell^*-W^*) Z_i W h,h \rangle\|\leq \|h\| \lim \|(W_\ell^*-W^*)\big(Z_i Wh\big)\|=0.
\eeq

To handle the remaining term in equation \eqref{eq:SplitSum} note that $Z_i W_\ell$ converges strongly to $Z_i W$ since $Z_i$ is compact. Also note that $\sup_\ell \|W_\ell^*\|\leq 1$ since the $W_\ell$ are contractions. Using these facts we have

\beq
\label{eq:WeakPlusCompactPart}
\lim \|\langle W_\ell^* Z_i (W_\ell-W)h,h \rangle\| \leq \|h\| \lim \|W_\ell^*\| \|\big(Z_i W_\ell-Z_i W\big) h\| =0.
\eeq
Combining equations \eqref{eq:StrongConvPart} and \eqref{eq:WeakPlusCompactPart} shows $\lim W_\ell^* Z_i W_\ell=W^* Z_i W$ for $i=1, \dots, g$.
\end{proof}

\begin{lemma}
\label{lem:ContractionToIso}
Let $X \in \mathfrak{K}(\cH)^g$ be a $g$-tuple of compact self-adjoint operators on $\cH$ and assume $0$ is in the finite interior of $K_X$. The there exists an integer $m_{0}$ depending only on $g$ such that, given any integers $m,n \in \bbN$ and any contraction $W:\cH_n \to \cH^m$, there exists an isometry $T: \cH_n \to \cH^{m+nm_0}$ such that
\[
W^*(I_m \otimes X) W=T^* (I_{m+nm_0} \otimes X ) T.
\]
\end{lemma}

\begin{proof}

By assumption $0$ is in the finite interior of $K_X$. It follows that there is an integer $m_0 \in \bbN$ and an isometry $Z_0: \cH_1 \to \cH^{m_0}$ such that $Z_0^* (I_{m_0} \otimes X) Z_0=0$. This implies that for each $n$ there is an isometry $Z_{0_n} : \cH_n \to \cH^{nm_{0}}$ such that $0_n=Z_{0_n}^* (I_{nm_{0}} \otimes X) Z_{0_n}$. Define $T:\cH_n \to \cH^{m+nm_{0}}$ by
\beq
T=\begin{pmatrix}
Z_{0_n} (I_n-W^*W)^\frac{1}{2}\\
W
\end{pmatrix}.
\eeq
Then $T$ is an isometry and $T^* (I_{m+nm_{0}} \otimes X) T=W^*(I_m \otimes X) W$, and the integer $m_{0}$ depends only on $g$.
\end{proof}

We now prove Proposition \ref{prop:FixedBoundOnXSize}.

\begin{proof}
Given $Y=V^* (I_\cH \otimes X) V \in K_X (n)$ where $V: \cH_n \to \cH^\infty$ is an isometry, let $P_\ell: \cH \to \cH$ be the orthogonal projection of $\cH$ onto $\cH_\ell$. Since $X$ is compact, the sequence $P_\ell X P_\ell$ converges to $X$ in norm, and $I_\cH \otimes P_\ell X P_\ell$ converges to $I_\cH \otimes X$ in norm. Defining
\beq
\label{eq:YellDef}
Y^\ell=V^*(I_\cH \otimes P_\ell X P_\ell) V
\eeq
it follows that $\lim Y^\ell =Y \in \bbS(\cH_n)^g$.

Observe that
\[
Y^\ell=((I_{\cH} \otimes \iota^*_\ell)V)^*(I_{\cH} \otimes (\iota_\ell^* X \iota_\ell))((I_{\cH} \otimes \iota^*_\ell)V)
\]
for all $\ell$ where $\iota_\ell: \cH_\ell \to \cH$ is the inclusion map. Defining the quantities
\[
X^\ell=\iota_\ell^* X \iota_\ell \in B(\cH_\ell)^g \quad \quad \mathrm{and} \quad \quad V_\ell=((I_{\cH} \otimes \iota_\ell^*)V) \in B(\cH_n, \cH_\ell^\infty)
\]
we have $V_\ell^* (I_\cH \otimes X^\ell) V_\ell=Y^\ell$ and $V_\ell$ is a contraction for all $\ell$.

Fix $d,\ell \in \bbN$ and a $g$-tuple $A \in \bbS(\cH_d)^g$. It is straight forward to show
\[
L_{Y^\ell} (A) \succeq  (V_\ell^* \otimes I_d) \big(I_\cH \otimes L_{X^\ell} (A)\big)(V_\ell^* \otimes I_d).
\]
It follows that $L_{Y^\ell}(A) \succeq 0$ if $L_{X^\ell} (A) \succeq 0$. Therefore $\cD_{X^\ell} \subseteq \cD_{Y^\ell}$ for all $\ell \in \bbN$.

Using \cite[Theorem 1.1]{HKM12} we conclude that there is an integer $m_n$ depending only on $n$ and $g$ such that for each $\ell \in \bbN$ there is contraction $W_\ell':\cH_n \to \cH_\ell^{m_n}$ such that
\beq
(W_\ell')^*(I_{m_n} \otimes X^\ell) W_\ell'=Y^\ell.
\eeq
For each $\ell$ define the operator $W_\ell: \cH_n \to \cH^{m_n}$ by
\beq
W_\ell=(I_{m_n} \otimes \iota_\ell)W_\ell'.
\eeq
Then each $W_\ell$ is a contraction and
\[
Y^\ell=W_\ell^*(I_{m_n} \otimes X) W_\ell
\]
for all $\ell \in \bbN$.

By passing to a subsequence if necessary, we may assume that the $W_\ell$ converge to some contraction $W:\cH_n \to \cH^{m_n}$ in the weak operator topology on $B(\cH_n, \cH^{m_n})$. By assumption $X$ is compact, so $I_{m_n} \otimes X$ is compact. Using Lemma \ref{lem:XCompactConvergence}, it follows that
\[
W^* (I_{m_n} \otimes X) W=\lim W^*_\ell (I_{m_n} \otimes X) W_\ell.
\]
It follows that $Y=W^* (I_{m_n} \otimes X) W$, which completes the proof of item \eqref{it:FiniteSumContraction}.

Item \eqref{it:FiniteSumIso} is immediate from Lemma \ref{lem:ContractionToIso} and the assumption that $0$ is in the finite interior of $K_X$.

\end{proof}

We now prove $K_X$ is a closed bounded matrix convex set.

\begin{theorem}
\label{theorem:KxClosedBoundedMatConvex}
Let $X\in \mathfrak{K}(\cH)^g$ be a $g$-tuple of self-adjoint compact operators on $\cH$ and let $K_X$ be the noncommutative convex hull of $X$. Assume that $0$ is in the finite interior of $K_X$. Then $K_X$ is a closed bounded matrix convex set.
\end{theorem}
\begin{proof}
Lemma \ref{lem:KxBoundedMatConvex} shows that $K_X$ is bounded and matrix convex. It remains to show that $K_X (n)$ is closed for each $n$. Let $\{Y^\ell\} \subset K_X (n)$ be a sequence of elements of $K_X (n)$ converging to some $g$-tuple $Y \in \bbS(\cH_n)^g$. By Proposition \ref{prop:FixedBoundOnXSize} there exists a fixed integer $m_n$ depending only on $n$ and $g$ and contractions $W_\ell:\cH_n \to \cH^{m_n}$ such that $W_\ell^* (I_{m_n} \otimes X) W_\ell=Y^\ell$ for all $\ell$. By passing to a subsequence if necessary, we can assume that the $W_\ell$ converge in the weak operator topology to some contraction $W:\cH_n \to \cH^{m_n}$. By assumption $X$ and $I_{m_n} \otimes X$ are compact so Lemma \ref{lem:XCompactConvergence} shows that
\[
Y=W^* (I_{m_n} \otimes X) W.
\]

Furthermore, we assumed $0$ is in the finite interior of $K_X$, so using Lemma \ref{lem:ContractionToIso} there exists an integer $m_{0}$ depending only on $g$ and an isometry $T:\cH_n \to \cH^{m_n+nm_{0}}$ such that
\[
T^* X T=W^*X W=Y.
\]
We conclude $Y \in K_X (n)$ and $K_X (n)$ is closed.
\end{proof}

\section{Matrix Affine Maps}
\label{section:MatrixAffine}

We begin this section by introducing and briefly discussing the notion of matrix affine maps on a matrix convex set. We direct the reader to \cite[Section 3]{WW99} for a more detailed discussion of matrix affine maps.

Let $K \subset \cup_n \bbS(\cH_n)^g$ be a closed bounded matrix convex set. A \df{matrix affine map} on $K$ is a sequence $\theta=(\theta_n)_n$ of mappings $\theta_n:K(n) \to M_n(W)$ for some vector space $W$, such that
\beq
\label{eq:MatrixAffineDef}
\theta_n \left(\sum_{\ell=1}^k V_\ell^* Y^\ell V_\ell\right)=\sum_{\ell=1}^k V_\ell^* \theta_{n_\ell} (Y^\ell) V_\ell,
\eeq
for all $Y^\ell \in K(n_\ell)$ and $V_\ell \in B(\cH_{n_\ell},M_n (W))$ for $\ell=1, \dots, k$ satisfying $\sum_{\ell=1}^k V_\ell^* V_\ell=I_n$. If each $\theta_n$ is a homeomorphism, then we will say $\theta$ is a \df{matrix affine homeomorphism}. Given a matrix convex set $K$ we will let
\[
\cR_K=\{\theta=(\theta_n)_n|\ \theta_n:K(n) \to B(\cH_n)\ \mathrm{\ for \ all\ } n \in \bbN \mathrm{\ and \ } \theta \mathrm{\ is \ matrix \ affine}\}
\] denote the set of matrix affine maps on $K$ sending $K(n)$ into $B(\cH_n)$. As an example, if $A \in B(\cH_d)^{g+1}$ is a $g+1$-tuple of operators on $\cH_d$, then the linear pencil $L_A$, i.e. the map $y \mapsto L_A(y)$, is an element of $M_d (\cR_K)$.

In \cite[Proposition 3.5]{WW99}, Webster and Winkler show  that the set $\cR_K$ is an operator system if $K \subset \cup_n \bbS(\cH_n)^g$ is a closed bounded matrix convex set. Given a positive integer $d$, the positive cone in $M_d(\cR_K)$ is defined by $\theta \in M_d (\cR_K)^+$ if and only if $\theta(Y) \succeq 0$ for all $Y \in K$.

Additionally, \cite[Proposition 3.5]{WW99} shows that, with these assumptions, $K$ is matrix affinely homeomorphic to $CS(\cR_K)$. In particular the map sending $Y\in K(n)$ to  $\phi_Y \in CS_n (\cR_K)$ where $\phi_Y$ is defined by
\beq
\label{eq:phiYdef}
\phi_Y(\theta)=\theta_n(Y) \mathrm{\ for\ all\ } \theta \in \cR_K
\eeq
 defines a matrix affine homeomorphism from $K$ to $CS(\cR_K)$. The identification between the matrix convex set $K$ and $CS(\cR_K)$ is strengthened by \cite[Theorem 4.2]{KLS14} where Kleski shows that $Y$ is an absolute extreme point of $K$ if and only if $\phi_Y \in CS(\cR_K)$ is an irreducible maximal completely positive map on $\cR_K$.

\subsection{Matrix Affine Maps on $K$ are Affine}

Webster and Winkler \cite{WW99} comment that, if all the elements of $K(1)$ are self-adjoint, as is the case in our setting, then the set of matrix affine maps on $K$ is equivalent to the set of affine maps on $K$. \cite{WW99} does not give a proof of this fact, as they do not use it, so for the reader's convenience we provide a proof here.

\begin{prop}
\label{prop:AffineArePencils}
Let $K$ be a closed bounded matrix convex set and assume $0 \in K$. Fix a $d \in \bbN$ and let $\theta \in M_d(\cR_K)$. Then there exists an
\[
A=(A_0, A_1, \dots, A_g) \in B(\cH_d)^{g+1}
\]
such that $\theta(Y)=L_A(Y)$ for all $Y \in K$.
\end{prop}
\begin{proof} Note that all the elements of $K$ are self-adjoint, so if $Y \in K$ and $\alpha \in \C$ satisfy $\alpha Y \in K$ then $\alpha \in \R$. Therefore, to show $\theta \in \cR_K$ is affine it is sufficient to show $\theta$ is affine over the reals. Temporarily assume $\theta \in \cR_K$ is self-adjoint. Define the map $\psi=(\psi_n)_n$ by
\[
\psi_n(Y)=\theta_n(Y) -\theta_n (0_n) \mathrm{\ for \ all \ } Y \in K (n) \mathrm{\  and\ all\ } n \in \bbN.
\]
We will show that $\psi$ is linear over $\R$ for each $n$.

Fix $n$ and a $Y \in K(n)$ and let $\alpha \in [0,1]$. Since $0 \in K$ and matrix convex sets are closed under taking direct sums we have $Y \oplus 0_n \in K (2n)$. Let $V: \C^n \to \C^{2n}$ be the isometry
\[
V=
\begin{pmatrix}
\sqrt{\alpha} I_n \\
\sqrt{1-\alpha} I_n
\end{pmatrix}
\]
and set $V_1=\sqrt{\alpha} I_n$ and $V_2=\sqrt{1-\alpha} I_n$. Then we have
\[
\begin{array}{rclcl}
\alpha \psi_n (Y)&=& \alpha \theta_n (Y)+(1-\alpha) \theta_n (0)-\theta_n(0) \\
&=& V_1^* \theta_n (Y) V_1+V_2^* \theta_n (0) V_2 -\theta_n (0) \\
&=& \theta_{n} (V^* (Y \oplus 0) V)-\theta_n (0) = \psi_n (\alpha Y).
\end{array}
\]

 Now let $\alpha>1$ and assume $\alpha Y \in K (n)$. Then $\frac{1}{\alpha} \in [0,1]$ so it follows that
$\frac{1}{\alpha} \psi_n (\alpha Y)= \psi_n (Y).$ This shows
\[
\psi_n (\alpha Y)=\alpha \psi_n (Y)
\]
for all $\alpha \geq 0$ satisfying $\alpha Y \in K$.

We now show that, given $Y^1, Y^2 \in K(n)$ such that $Y^1+Y^2 \in K(n),$ we have
\[
\psi_n(Y^1+Y^2)=\psi_n(Y^1+Y^2).
\]
To this end, set
\[
V=\begin{pmatrix}
\frac{1}{\sqrt{2}} I_n \\
\frac{1}{\sqrt{2}} I_n
\end{pmatrix}.
\]
Since $\psi$ is matrix affine we have
\[
\begin{array}{rclcl}
\psi_n (\frac{1}{2} Y^1)+\psi_n(\frac{1}{2} Y^2) &=& \frac{1}{2} \psi_n (Y^1)+\frac{1}{2} \psi_n (Y^2) \\
&=& V_1^* \psi_n (Y^1) V_1+V_2^* \psi_n(Y^2) V_2 \\
&=& \psi_{2n} (V^* (Y^1 \oplus Y^2) V) \\
&=& \psi_{n} (\frac{1}{2} Y^1+\frac{1}{2} Y^2).
\end{array}
\]
By assumption $Y^1+Y^2 \in K$ so we find
\[
\psi_n (Y^1)+\psi_n(Y^2)=2\left(\psi_n (\frac{1}{2} Y^1)+\psi_n(\frac{1}{2} Y^2) \right)
=2\psi_n (\frac{1}{2} Y^1+\frac{1}{2} Y^2)
= \psi_n (Y^1+Y^2).
\]
Additionally, given $Y \in K(n)$ such that $-Y \in K$ we have
\[
0_n=\psi(0_n)=\psi_n(Y-Y)=\psi_n(Y)+\psi_n(-Y).
\]
Therefore $\psi_n(-Y)=-\psi_n(Y)$. We conclude that $\psi_n$ is linear for each $n$ and $\theta_n$ is affine for each $n$.

Now recall that we are dealing with self-adjoint $\theta$. In particular $\theta_1$ is affine and self-adjoint, so there exists a $g+1$-tuple $(\alpha_0, \alpha_1, \dots, \alpha_g) \in \R^{g+1}$ such that $\theta_1 (Y)=\alpha_0-\sum_{i=1}^g \alpha_i Y_i$ for all $Y \in K (1)$. Since $\theta$ is self-adjoint and matrix affine, each $\theta_n$ is determined by the equality
\[
\langle \theta_n (Y) \zeta,\zeta \rangle=\theta_1(\zeta^* Y \zeta)
\]
for all $Y \in K (n)$ and all unit vectors $\zeta \in \C^n$ and all $n$. It follows that
\[
\theta_n (Y)=\alpha_0 I_n-\sum_{i=1}^g \alpha_i Y_i
\]
for all $Y \in K (n)$ and all $n$.

Now if $\theta \in \cR_K$ is not self-adjoint then $\theta$ can be written $\theta=\theta^1+i \theta^2$ where $\theta^1$ and $\theta^2$ are self-adjoint. It follows from above that there is a $g+1$ tuple $(\alpha_0, \alpha_1, \dots, \alpha_g) \in \C^{g+1}$ such that
\beq
\label{eq:FAffine}
\theta_n (Y)=\alpha_0 I_n-\sum_{i=1}^g \alpha_i Y_i
\eeq
for all $Y \in K (n)$ and all $n$.

It immediately follows that if $\theta \in M_d (\cR_K)$ then there exists a tuple $A=(A_0,A_1, \dots, A_g) \in M_d(\C^{g+1})$ such that
\[
\theta(Y)=A_0 \otimes I_n - \sum_{i=1}^g A_i \otimes Y_i=L_A(Y)
\]
 for all $Y \in K (n)$ and all $n$. Identifying $M_d(\C^{g+1})$ with $B(\cH_d)^{g+1}$ for each $d$ completes the proof.

\end{proof}

In light of Proposition \ref{prop:AffineArePencils}, if $0 \in K$, then for a fixed $d \in \bbN$ we have
\beq
\label{eq:AffineSetIsPencilsSet}
M_d(\cR_K)=\{L_A:K \to M_d(K)| \ A \in B(\cH_d)^{g+1}\}
\eeq
where $L_A$ is the map $y \mapsto L_A(y).$

\begin{remark}
As an aside for the reader interested in polar duals, we note that Proposition \ref{prop:AffineArePencils} points towards a strong relationship between positive cone in $\cup_d M_d(\cR_K)$ and $K^\circ$, the polar dual of $K$. In particular, given a tuple $A=(A_0, A_1, \dots, A_g) \in B(\cH_d)^{g+1}$ it can be shown that $L_A \in M_d(\cR_K)^+$ if and only if there exists a tuple $\tilde{A} \in K^\circ (m)$ for some $m \leq d$ and a positive definite operator $\tilde{A}_0 \in B(\cH_m)$ such that $A$ is unitarily equivalent to the tuple
\[
(\tilde{A}_0 \oplus 0_{d-m}, -\tilde{A}_{0}^{1/2} \tilde{A}_1 \tilde{A}_0^{1/2} \oplus 0_{d-m}, \dots, -\tilde{A}_{0}^{1/2} \tilde{A}_g \tilde{A}_0^{1/2}\oplus 0_{d-m}).
\]
We omit the proof of this fact as we will not make use of the fine structure of the positive cone in $\cup_d M_d(\cR_K)$. See \cite{HKM17} for a general discussion of polar duals and \cite{EHKM16} for a discussion of the extreme points of polar duals of a free spectrahedra.
\end{remark}

\subsection{Maps on $\cR_K$}

Given a Hilbert space $\cM$ and a g-$tuple$ of operators $Z \in B(\cM)^g$ we define the map $\phi_Z:\cR_K \to B(\cM)$ by
\beq
\label{eq:phiZdefPencil}
\phi_Z(L_A)=L_A(Z) \mathrm{ \ for\ all\ pencils\ } L_A \in \cR_K.
\eeq

We are particularly interested in the case where $K=K_X$ for some g-tuple of self-adjoint compact operators $X \in \ktuples$. The following proposition shows that the map $\phi_X:\cR_{K_X}\to B(\cH)$ is a unital completely positive map on $\cR_{K_X}$ when $0$ is in the finite interior of $K_X$.

\begin{prop}
\label{prop:phiXisCP}

Let $X \in \mathfrak{K}(\cH)^g$ and assume $0$ is in  the finite interior of $K_X$. Then $\phi_X: \cR_{K_X}\to B(\cH)$ as defined by equation \eqref{eq:phiZdefPencil} is a unital completely positive map on $\cR_{K_X}$.
\end{prop}
\begin{proof}

By assumption $0$ is in the finite interior of $K_X$. Therefore, Proposition \ref{prop:AffineArePencils} shows that $\cR_{K_X}$ is equal to the set of linear pencils on $K_X$.

 For all integers $n \in \bbN$ let $P_n: \cH \to \cH$ be the orthogonal projection of $\cH$ onto $\cH_n$, and define $X^n \in \mathfrak{K}(\cH)^g$ to be the tuple $X^n=P_n X P_n$. Observe that $X^n=\iota_n^* X \iota_n \oplus 0_{\cH_n^\perp}$ where $\iota_n:\cH_n \to \cH$ is the inclusion map of $\cH_n$ into $\cH$.

From the definition of $K_X$ we know that $\iota^*_n X \iota_n \in K_X$. Since $0,\iota^*_n X \iota_n \in K_X$, using \cite[Proposition 3.5]{WW99} we find $\phi_0,\phi_{\iota_n^* X \iota_n} \in CS(\cR_{K_X})$.  As such, the equality
 \beq
 \label{eq:phiXnEqualityOne}
 \phi_{X^n}(L_A)=L_A(\iota_n^* X \iota_n \oplus (I_{\cH_n^\perp} \otimes 0))=L_A(\iota_n^* X \iota_n) \oplus \big(I_{\cH_n^\perp} \otimes L_A(0)\big)=\phi_{\iota_n^* X \iota_n} \oplus \big(I_{\cH_n^\perp} \otimes \phi_{0}(L_A)\big)
 \eeq
for all $L_A \in \cR_{K_X}$ shows that $\phi_{X^n}$ is completely positive for all $n \in \bbN$.

Now fix $d \in \bbN$ and a $L_A \in M_d(\cR_{K_X})$. Since $X$ is compact we have $\lim X^n=X$ where the convergence is in norm. Furthermore, linear pencils are continuous maps, so
\[
\lim \phi_{X^n}(L_A)=\lim L_A(X^n)=L_A(X)=\phi_X(L_A).
\]
Since each $\phi_{X^n}$ is completely positive, it follows that $\phi_X$ is completely positive on $\cR_{K_X}$ as claimed.

To see that $\phi_X$ is unital let $1_\cR$ be the identity in $\cR_{K_X}$. Then $1_\cR$ is the linear pencil
\[
1_\cR=L_{(1,0,0, \dots, 0)}.
\]
The evaluation
\[
\phi_X(1_\cR)=L_{(1,0,0, \dots, 0)}(X)=I_\cH
\]
shows $\phi_X$ is unital.

\end{proof}

\section{Absolute Extreme Points of Noncommutative Convex Hulls}
\label{section:NoAbsExtreme}

We are almost in position to prove our main result. We first give a lemma that describes the reducing subspaces of a direct sum of a fixed $g$-tuple with itself.

\begin{lem}
\label{lem:InvarSubspaces}
Let $\cH$ be an infinite dimensional Hilbert space and let $X \in \bbS(\cH)^g$ be a $g$-tuple of self-adjoint operators on $\cH$. Assume that every nontrivial reducing subspace of $X$ is infinite dimensional. Then, for any integer $N \in \bbN$, every nontrivial reducing subspace of $I_N \otimes X \in \bbS(\cH^N)^g$ is infinite dimensional.

\end{lem}

\begin{proof}
Fix an integer $N \in \bbN$ and let $W \subset \cH^N$ be any reducing subspace for $I_N \otimes X$. For each $n=1, \dots, N$ define $\cM_n=\cH$. Then $\cH^N=\oplus_{n=1}^N \cM_n$. For each $n=1, \dots, N$ let $\iota_n:\cM_n \to \cH^N$ be the inclusion map of $\cM_n$ in $\cH^N$. Given $v \in W$, define $v_n \in \iota_n^* W$ by $v_n=\iota_n^* v$ for all $n=1, \dots, N$. Then $v$ can be written $v=\oplus_{n=1}^N v_n$. Observe that $(I_N \otimes X_i) v=\oplus_{n=1}^N (X_i v_n)$ for all $i=1, \dots, g$. Since $W$ is a reducing subspace for $I_N \otimes X,$ it follows that $\oplus_{n=1}^N (X_i v_n) \in W$ for all $v \in W$ and all $i=1, \dots, g$. Fix an $n_0 \in \{1, \dots, N\}$. Then applying $\iota^*_{n_0}$ to both sides of the equality we find
\[
X_i v_{n_0}=\iota^*_{n_0}\big(\oplus_{n=1}^N (X_i v_n)\big) \in \iota^*_{n_0} W
\]
for all $v_{n_0} \in \iota^*_{n_0} W$ and all $i=1, \dots, g$. Since $X$ is a tuple of self-adjoint operators, this shows that $\iota^*_{n_0} W$ is a reducing subspace for $X$ for all $n_0 \in \{1, \dots, N\}$. As $X$ was assumed to have no nontrivial finite dimensional reducing subspaces, it follows that $\iota^*_n W=\{0\}$ or $\iota^*_n W$ is infinite dimensional for all $n=1, \dots, N$. If $\iota^*_n W=\{0\}$ for all $n=1, \dots, N$ then $W=\{0\}$. If $\iota^*_n W\neq \{0\}$ for any $n$ then $W$ is infinite dimensional.
\end{proof}

We now complete the proof of Theorem \ref{thm:NoAbsExt}.

\begin{proof}

Theorem \ref{theorem:KxClosedBoundedMatConvex} shows that $K_X$ is a closed bounded matrix convex set, so we only need to show $\abs{K_X}=\emptyset$. Pick an element $Y \in K_X (n)$ and let $\maxRKX$ be the operator system of matrix affine maps on $\maxK_X$. \cite[Theorem 4.2]{KLS14} shows that $Y \in \abs K_X$ if and only if $\phi_Y$ is in the Arveson boundary of $CS(\cR_K)$ and \cite[Proposition 2.4]{A08} shows that if $\phi_Y$ is a boundary representation, then $\phi_Y$ is maximal. Therefore, to show $Y \notin \abs K_X$ it is sufficient to show that the unital completely positive map $\phi_Y: \cR_{K_X} \to \cH_n$ is not maximal.

Using Proposition \ref{prop:FixedBoundOnXSize} there exists an integer $m_n$ depending only on $n$ and $g$ and an isometry $V: \cH_n \to \cH^{m_n}$ so that $Y=V^*(I_{m_n} \otimes X) V$. This implies that there is a unitary $U:\cH^{m_n} \to \cH^{m_n}$ so that $U (I_{m_n} \otimes X) U^*$ is a dilation of $Y$. It follows that the map
\[
\phi_{U(I_{m_n} \otimes X)U^*}:\cR_{K_X} \to B(\cH^{m_n})
\] is a dilation of the unital completely positive map $\phi_Y:\cR_{K_X} \to \cH_n$. Furthermore, Proposition \ref{prop:phiXisCP} shows $\phi_X$ is a unital completely positive map on $\cR_{K_X}$, so the equality
\[
\phi_{U(I_{m_n} \otimes X)U^*} (L_A)=U\big(I_\cH \otimes \phi_X(L_A)\big)U^*
\]
for all $L_A \in \cR_{K_X}$ shows $\phi_{U(I_{m_n} \otimes X)U^*}$ is a unital completely positive map on $\cR_{K_X}$.

Assume towards a contradiction that $\phi_Y$ is a maximal unital completely positive map. Since $\phi_{U (I_{m_n} \otimes X) U^*}$ is a unital completely positive dilation of $\phi_Y$, there must exist some unital completely positive map $\psi:\cR_{K_X} \to \cH_n^\perp $ so that $\phi_{U(I_{m_n} \otimes X)U^*}=\phi_Y \oplus \psi$. Note that in this definition, $\cH_n$ is viewed as a subspace of $\cH^{m_n}$, so $\cH_n^\perp$ is the orthogonal complement of $\cH_n$ in $\cH^{m_n}$.

For $i=1, \dots, g$ let $\eta_i \in \maxRKX$ be the evaluation map defined by
\[
\phi_S (\eta_i)=S_i \mathrm{\ for\ all\ } S \in \maxK_X.
\]
Define $Z \in \bbS(\cH_n^\perp)^g$ to be the tuple
\[
Z=(\psi (\eta_1), \dots, \psi (\eta_g)).
\]

Considering the evaluations
\[
\phi_{U(I_{m_n} \otimes X)U^*} (\eta_i)= \phi_Y (\eta_i) \oplus \psi (\eta_i)
\]
for $i=1, \dots, g$ shows that $I_{m_n} \otimes X=U^* (Y \oplus Z) U$. In particular, the invariant subspaces of $I_{m_n} \otimes X$ must be equal to the invariant subspaces of $U^*(Y \oplus Z)U$.

Since $X$ is assumed to have no nontrivial finite dimensional reducing subspaces, Lemma \ref{lem:InvarSubspaces} shows that any nontrivial reducing subspace of $I_{m_n} \otimes X$ is infinite dimensional. Observe that the subspace $W \subset \cH^{m_n}$ defined by
\[
W=\{U^*(v \oplus 0_{{\cH_n}^\perp}) |\ v \in \cH_n\}
\]
is a nontrivial $n$ dimensional reducing subspace for $U^*(Y \oplus Z) U$, and hence for $I_{m_n} \otimes X,$ which contradicts Lemma \ref{lem:InvarSubspaces}. It follows that there is no unital completely positive map $\psi$ so that $\phi_Y \oplus \psi=\phi_{U (I_{m_n} \otimes X) U^*}$. In particular, $\phi_Y$ is not maximal. This shows that $Y \notin \abs K_X$ and $\abs K_X=\emptyset$.

\end{proof}

\section{Examples}
\label{section:ExampleX}

\def\frakI{\mathfrak{I}}

The following section gives an explicit example of a tuple $X \in \mathfrak{K}(\cH)^2$ of compact operators with no nontrivial finite dimensional reducing subspaces so that $0$ is in the finite interior of $K_X$. Throughout the section set $\cH=\ell^2(\bbN)$ and $\cH_n=\ell^2(1, \dots, n) \subset \cH $ for all $n \in \bbN$.

Given a weight vector $w=(w_1, w_2, \dots) \in \R^\infty$ define the weighted forward shift $S_w: \cH \to \cH$ by
\[
S_w v=(0,w_1 v_1, w_2 v_2, \dots)
\]
for all $v \in \cH$. Additionally, for each $n \in \bbN$, let $\mathfrak{I}_n: \cH \to \cH$ be the operator defined by
\[
\mathfrak{I}_n v=(v_1, \dots, v_n, 0,0, \dots)
\]
for all $v \in \cH$.

\begin{prop}
\label{prop:ExampleX}
Let $X_1=\operatorname{diag}(\lambda_1, \lambda_2, \dots)$ where the $\lambda_i$ nonzero real numbers converging to $0$ with distinct norms and let $S_w$ be a weighted shift where $w \in \R^\infty$ is chosen so $w_i \neq 0$ for all $i$ and $S_w$ is compact. Set
\[
X_2={S_w+S_w^*}.
\]
Then there exists real numbers $\alpha_1, \alpha_2$ so that the tuple
\[
\tilde{X}=(X_1+\alpha_1 \frakI_2,X_2+\alpha_2 \frakI_2)
\]
is a tuple of compact self-adjoint operators with no finite dimensional reducing subspaces and so that $0$ is in the finite interior of $K_{\tilde{X}}$.
\end{prop}

Before giving the proof of Proposition \ref{prop:ExampleX}, we state a lemma which describes the invariant subspaces of a diagonal operator.

\begin{lem}
\label{lem:DiagInvarSubspaces}
Let $X=\mathrm{diag}(\lambda_1, \lambda_2, \dots)$ where the $\lambda_i$ nonzero real numbers converging to $0$ with distinct norms, and let $W$ be a closed invariant subspace of $X$. Then $W=\oplus_{j\in J} E_j$ for some index set $J \subset \bbN$ where $E_j$ denotes $j$th coordinate subspace.
\end{lem}
\begin{proof}
If $W=\{0\}$ then the proof is trivial, so assume $W \neq \{0\}$. Define
\[
J=\{j \in \bbN| \ \mathrm{ there \ exists \ a \ vector\ } v=(v_1, v_2, \dots) \in W \mathrm{\ so \ that\ } v_j \neq 0\}.
\]
Since $W \neq \{0\}$ we have $J \neq \emptyset$. We will show $W=\oplus_{j\in J} E_j$.

By assumption the $\lambda_i$ have distinct norms and converge to zero so there is a unique index $j_0 \in J$ such that $|\lambda_{j_0}|=\max_{j \in J} |\lambda_j|$. Choose a vector $v \in J$ so that $v_{j_0} \neq 0$. Then
\[
\lim_{\ell \to \infty} \frac{X^\ell v}{\lambda_{j_0}^\ell}=v_{j_0} e_{j_0}.
\]
Therefore $e_{j_0} \in W$ since $W$ is closed.

Since $E_{j_0}$ and $W$ are closed invariant subspaces of $X$ with $E_{j_0} \subset W$, it follows that $W=E_{j_0} \oplus W'$ where $W'$ is a closed invariant subspace of $X$. Proceeding by induction completes the proof.
\end{proof}

We now prove Proposition \ref{prop:ExampleX}.

\begin{proof}

We first prove the existence of the real numbers $\alpha_1, \alpha_2$ so that $0$ is in the finite interior of $K_X$. Let $\iota_2:\cH_2 \to \cH$ be the inclusion map of $\cH_2 \to \cH$. Since $\lambda_1 \neq \lambda_2$, there exists a unit vector $v_0 \in \cH_2$ so that, setting \[
\alpha_1= -\langle \iota^*_2 X_1 \iota_2 v_0,v_0 \rangle,
\]
the eigenvalues of
\[
X_1+\alpha_1 \frakI_2=\operatorname{diag} (\lambda_1+\alpha_1, \lambda_2+\alpha_1,\lambda_3,\lambda_4,\dots)
\]
are nonzero real numbers with distinct norms.

Set
\[
\alpha_2= -\langle \iota^*_2 X_2 \iota_2 v_0,v_0 \rangle.
\]
Then
\[
v_0^*(\iota_2 X_1 \iota_2+\alpha_1 I_2,\iota_2 X_2 \iota_2+\alpha_2 I_2)v_0=0 \in \R^2.
\]
Setting $\tilde{X}_i=X_i+\alpha_i \frakI_2$ for $i=1,2$, it follows that
\beq
\label{eq:ThePropertyWeWant}
v_0^*  (\iota_2^* \tilde{X}_1 \iota_2, \iota_2^*\tilde{X}_2 \iota_2)v_0=0 \in \R^2.
\eeq
Therefore, $0$ is in the finite interior of $\tilde{X}=(\tilde{X}_1, \tilde{X}_2)$.

It is clear that $\tilde{X}$ is a tuple of compact self-adjoint operator, so it remains to show that $\tilde{X}$ has no finite dimensional reducing subspaces. Let $W$ be a finite dimensional reducing subspace for $\tilde{X}$. Then $W$ must be a closed invariant subspace of $\tilde{X}_1$. Recall that $\alpha_1$ was chosen so that $\tilde{X}_1$ is a diagonal operator whose diagonal entries are real numbers that converge to $0$ with distinct norms. Using Lemma \ref{lem:DiagInvarSubspaces}, it follows that $W=\oplus_{j \in J} E_j$ for some finite index set $J \subset \bbN$ where $E_j$ denotes $j$th coordinate subspace.

Let $j_0$ be the largest integer in $J$. Since $S_w$ is a weighted forward shift with nonzero weights, it is straight forward to see that
\[
\tilde{X}_2 E_{j_0}=\left(S_w+S_w^*+\alpha_2 \frakI_2\right) E_{j_0} \not\subset W.
\]
This shows that $W$ cannot be a reducing subspace of $\tilde{X}$. Therefore $\tilde{X}$ has no finite dimensional reducing subspaces.

\end{proof}

\section*{Acknowledgments}
I thank my Ph.D. advisor, Bill Helton, for many enlightening discussions related to this article and for tremendous help in the preparation of the paper. Comments by Scott McCullough and Igor Klep were also greatly appreciated.

\newpage

NOT FOR PUBLICATION

\tableofcontents

\newpage

\printindex


\begin{thebibliography}{1}

\bibitem[A69]{A69} W. Arveson:
{\it Subalgebras of $C^*$-algebras}, Acta Math. {\bf 123} (1969) 141-224.

\bibitem[A72]{A72} W. Arveson:
{\it Subalgebras of $C^*$-algebras, II}, Acta Math. {\bf 128}  (1972) 271-308.

\bibitem[A08]{A08} W. Arveson: {\it The noncommutative Choquet boundary}, J. Amer. Math. Soc. {\bf 21} (2008) 1065-1084.

\bibitem[DDSS17]{DDSS17}
K.R. Davidson, A. Dor-On, M. Orr, B. Solel:
{\it Dilations, inclusions of matrix convex sets, And completely positive maps},
Int. Math. Res. Not. {\bf 13} (2017) 4069-4130.

\bibitem[DK15]{DK15} K.R. Davidson, M. Kennedy:
{\it The Choquet boundary of an operator system}, Duke Math. J. {\bf 164} (2015) 2989-3004.

\bibitem[DM05]{DM05} M.A. Dritschel, S.A McCullough:
{\it Boundary representations for families of representations of operator algebras and spaces}, J. Operator Theory {\bf 53} (2005) 159-168.

\bibitem[EHKM17]{EHKM16}
E. Evert, J.W. Helton, I. Klep, S. McCullough:
{\it Extreme Points of Matrix Convex Sets, Free Spectrahedra and Dilation Theory},
J. Math. Anal. Appl. {\bf 445} (2017) 1047-1070.

\bibitem[F00]{F00} D.R. Farenick:
{\it Extremal Matrix states on operator systems}, J. London Math. Soc. {\bf 61} (2000)885-892.

\bibitem[F04]{F04} D.R. Farenick:
{\it Pure matrix states on operator systems}, Linear Algebra Appl. {\bf 393} (2004) 149-173.

\bibitem[FHL+]{FHL+}
A.H. Fuller, M. Hartz, M. Lupini
{\it Boundary representations of operator spaces, and compact rectangular matrix convex sets}, preprint https://arxiv.org/abs/1610.05828

\bibitem[Ham79]{Ham79}
M. Hamana:
{\it Injective envelopes of operator systems},
 Publ. Res. Inst. Math. Sci. {\bf 15} (1979) 773-785.

\bibitem[HKM12]{HKM12}
J.W. Helton, I. Klep, S. McCullough:
{\it The convex Positivstellensatz in a free algebra},
Adv. Math. {\bf 231} (2012) 516-534.

\bibitem[HKM17]{HKM17}
J.W. Helton, I. Klep, S. McCullough:
{\it The tracial Hahn-Banach theorem, polar duals, matrix convex sets, and projections of free spectrahedra}
J. Eur. Math. Soc. {\bf 19} (2017) 1845-1897.

\bibitem[KLS14]{KLS14}
C. Kleski: {\it Boundary representations and pure completely positive maps}, J. Operator Theory {\bf 71} (2014) 45-62.

\bibitem[MS98]{MS98}
P.S. Muhly, B. Solel:
" An algebraic characterization of boundary representations" In {\it Nonselfadjoint Operator Algebras, Operator Theory, and Related Topics}, Oper. Theory Adv. Appl. {\bf 104}, Birk\"{a}user, Basel, 1998, 189-196

\bibitem[P02]{P02}
V. Paulsen:
{\it Completely bounded maps and operator algebras},
Cambridge University Press, 2002.

\bibitem[WW99]{WW99}
C. Webster and S. Winkler:
{\it The Krein-Milman Theorem in Operator Convexity},  Trans Amer. Math. Soc. {\bf 351} (1999) 307-322

\end{thebibliography}
\end{document}